%% file: LDP.tex
\documentclass[11pt,a4paper,reqno]{amsart}
\usepackage[english]{babel}
\usepackage[T1]{fontenc}
\usepackage[utf8]{inputenc}
\usepackage{csquotes}

\input{preamble/packages}
\input{preamble/shortcuts}

\input{preamble/theoremEnvironments}

\usepackage{hyperref}
\title[Sample-path LDP for Scheduled arrival processes with Unpunctuality]{Sample-path Large deviations for \\Scheduled Arrival Processes with Unpunctuality }
\author{Serena Della Corte}
\address{University of Leiden, Gorlaeus Building, Einsteinweg 55 2333 CC Leiden, }
\email{s.della.corte@math.leidenuniv.nl}

\author{Michel Mandjes}
\address{University of Leiden, Gorlaeus Building, Einsteinweg 55 2333 CC Leiden, }
\email{m.r.h.mandjes@math.leidenuniv.nl}
\begin{document}
\maketitle

\begin{abstract}
    We study a scheduled-arrival process in which deterministic arrival times are perturbed by i.i.d. unpunctuality random variables. For the empirical mean of independent copies of the resulting counting process, we prove a sample-path large deviations principle on $D([0,T])$, equipped with the Skorokhod topology. As an application, we derive large deviation principles for the corresponding workload process under deterministic and random service setting. In the deterministic service case, we also identify the exponentially tilted arrival path associated with a rare workload overflow event. \smallskip

     \smallskip


\end{abstract}

\section{Introduction}
In many real service systems, arrivals are planned in advance rather than occurring spontaneously. Patients are assigned appointment times at hospitals, customers reserve time slots for services, and travelers are scheduled for check-in or security screening. In practice, however, scheduled arrival times are rarely met exactly: some customers arrive early, others late, and these deviations from schedule determine the actual congestion experienced by the system. A familiar example is provided by the large COVID-19 vaccination campaigns, where individuals were given fixed vaccination appointments, yet the queues that formed were driven largely by the randomness in their arrival times relative to those appointments rather than by the appointment schedule itself.

Such systems are not well captured by the classical Poisson arrival model. On the one hand, arrivals are not completely random, since they are anchored to an underlying schedule. On the other hand, they are not deterministic either, because the scheduled times are subject to random deviations reflecting individual unpunctuality. The resulting arrival process therefore combines a deterministic scheduling mechanism with stochastic perturbations. This provides the motivation for studying scheduled arrival processes with random perturbations, which bridge the gap between purely random and perfectly scheduled arrivals.

In this paper, we consider a deterministic schedule with equally spaced appointment times $k\Delta$, where $k\in\Z$ and $\Delta>0$ is fixed. The actual arrival time associated with slot~$k$ is
\begin{equation}
\label{eq:intro_arrival_time}
A_k: = k\Delta + T_k,
\end{equation}
where the random variables $(T_k)_{k\in\Z}$ are independent and identically distributed. The corresponding arrival counting process is
\[
A(t)=\sum_{k\in\Z}\mathbf{1}_{\{0<A_k\le t\}}, \qquad t\in[0,T].
\]
The deterministic schedule specifies the intended arrival times, while the random perturbations $T_k$ capture customers' unpunctuality, allowing for both early and late arrivals. This simple construction separates the planned component of the arrival process from the stochastic variability around it.

\medskip

When many customers arrive substantially earlier or later than scheduled, an unusually large number of arrivals may accumulate within a short time interval, creating a rare congestion event. Large deviations theory provides a natural framework for quantifying the probability of such events and for identifying the most likely evolution of the arrival process leading to congestion.
Large deviations techniques have been widely applied to the analysis of rare events in queueing systems, particularly those involving unusually large workloads or queue lengths. For example, \cite{DuffieldOConnell} used large deviations principles to derive overflow asymptotics for single-server queues, while \cite{BotvichDuffield} analyzed buffer overflow and loss probabilities in systems where many (essentially independent) input streams share a common buffer. These results demonstrate the effectiveness of large deviations theory in characterizing rare but operationally important events in stochastic service systems.

This many-sources viewpoint is particularly close to the setting of the present paper: the input is formed by aggregating many independent sources, and one studies how rare congestion events arise as the number of sources increases. Sample-path large deviations for queues with many inputs were developed in \cite{Wischik01}, in a very general traffic framework allowing for broad classes of arrival processes, and a broader treatment of large deviations methods for queueing systems can be found in \cite{BigQueues04}, which develops a general theory beyond specific many-sources scaling regimes. In contrast, the related work \cite{Mandjes1999} focuses on identifying the most likely path leading to buffer overflow in a more specific setting, while \cite{ManjesKim} studies a related many-sources regime in the context of small-buffer asymptotics.

\medskip

As discussed above, the input process considered in this paper has a particular structure that arises naturally in appointment-based systems. Arrivals follow a deterministic schedule, with randomness introduced through customer unpunctuality, which acts as a perturbation of the scheduled times. Appointment systems with such timing perturbations have been studied from a performance-analysis perspective; see, for example, \cite{FiemsDeVuyst,DefeFiemsDeVuyst}. A closely related recent contribution is \cite{DeTurckFiems26}, where the authors show that, although the arrival process may be well approximated by a Poisson process, this approximation can fail at the queueing level.

Our focus differs from this performance-analysis perspective. We study the scheduled-arrival process itself from a large deviations viewpoint and, to the best of our knowledge, establish the first sample-path large deviations principle for the empirical mean of such counting processes. The proof combines finite-dimensional large deviations, exponential tightness, and a variational characterization of the rate function. A key feature is that it relies only on boundedness of the periodicized density associated with the unpunctuality distribution (Assumption~\ref{ass:Per}), an assumption verified for Laplace and double-Pareto unpunctuality distributions.

As an application, we study workload overflow. We first consider deterministic service and then extend the analysis to random service times. In both cases, the workload LDP follows from the arrival-process LDP via the contraction principle. For deterministic service, we also identify the exponentially tilted arrival path associated with a rare overflow event.

\medskip

The rest of the paper is organized as follows. Section~\ref{sec:Preliminaries} recalls the main definitions used in the paper. Section~\ref{sec:Model} introduces the scheduled-arrival model and the assumption. Section~\ref{sec:log-laplace} derives the relevant log-Laplace functionals. Section~\ref{sec:main_theorem} proves the sample-path large deviations principle. In Section~\ref{sec:examples}, we verify the main assumption for two examples of unpunctuality distributions. Section~\ref{sec:workload} applies the result to workload large deviations and rare overflow paths.

\section{Preliminaries}\label{sec:Preliminaries}
In this section, we give the main definitions and general results used throughout the paper. 
\subsection{Main definitions}

We denote by $D([0,T])$ the space of càdlàg functions $\phi:[0,T]\to\R$, i.e., functions that are right-continuous with left limits. 
We equip $D([0,T])$ with the Skorokhod $J_1$ topology. We refer to \cite{EK86, Bill99} for background on this topology.

The following modulus of continuity plays a central role in tightness criteria for processes in $D([0,T])$.

\begin{definition}[Skorokhod modulus]\label{def:wprime}
For $\phi \in D([0,T])$ and $\delta>0$, define
\begin{equation}
w'(\phi,\delta,T)
:=
\inf_{\Pi}
\max_{1\le i\le m}
\sup_{s,t \in [t_{i-1}, t_i]} |\phi(t)-\phi(s)|,
\end{equation}
where the infimum is taken over all partitions
$
\Pi=\{0=t_0<t_1<\cdots<t_m=T\}
$
such that
$
t_i - t_{i-1} \ge \delta$, with $i=1,\dots,m.$
\end{definition}

The quantity $w'(\phi,\delta,T)$ measures the maximal oscillation of the path $\phi$ over time intervals of length at least $\delta$, optimized over all such partitions. It is a standard modulus used in compactness and tightness criteria in the Skorokhod space; see \cite{FK06,EK86}.

For completeness, we briefly recall the notion of a {\it large deviations principle} (LDP); see \cite{DZ98,FK06}.

\begin{definition}[Large deviations principle]
Let $(E,\mathcal{T})$ be a topological space. A sequence of random elements $\{X^{(n)}\}$ taking values in $E$ satisfies a large deviations principle with speed $n$ and rate function $I:E\to[0,\infty]$ if:
\begin{itemize}
\item for every closed set $F\subset E$,
\[
\limsup_{n\to\infty}\frac{1}{n}\log \Pbb(X^{(n)}\in F)
\le -\inf_{x\in F} I(x),
\]
\item for every open set $G\subset E$,
\[
\liminf_{n\to\infty}\frac{1}{n}\log \Pbb(X^{(n)}\in G)
\ge -\inf_{x\in G} I(x).
\]
\end{itemize}
The rate function $I$ is called {\it good} if its level sets $\{x: I(x)\le M\}$ are compact for all $M<\infty$.
\end{definition}
The notion of {\it exponential tightness} \cite{DZ98,FK06} is central in establishing LDPs in infinite-dimensional spaces.

\begin{definition}[Exponential tightness]
A sequence of random elements $\{X^{(n)}\}$ in $D([0,T])$ is said to be exponentially tight with speed $n$ if for every $M<\infty$, there exists a compact set $K_M \subset D([0,T])$ such that
\[
\limsup_{n\to\infty} \frac{1}{n} \log \Pbb(X^{(n)} \notin K_M)
\le -M.
\]
\end{definition}




\section{Model and Assumptions}\label{sec:Model}
We introduce the scheduled-arrivals model with random unpunctuality and state the structural assumptions used throughout the paper.

\subsection{Scheduled Arrivals with Unpunctuality}
Fix $\Delta>0$. For each $k\in\Z$, let
\begin{equation}\label{eq:Ak}
A_k := k\Delta + T_k,
\end{equation}
where $(T_k)_{k\in\Z}$ are i.i.d.\ continuous random variables with distribution function $F$ and density $f$. Thus, customers are scheduled to arrive at the deterministic times $k\Delta$, but experience independent random deviations.
Given a time horizon $T>0$, define the arrival counting process
\begin{equation}\label{eq:At}
A(t):=A((0,t])=\sum_{k\in\Z}\1_{\{0<A_k\le t\}},\qquad t\in[0,T].
\end{equation}
The process $A$ is nondecreasing with càdlàg sample paths.
For each $k\in\Z$, the distribution function and density of $A_k$ are
\begin{equation}\label{eq:Pk}
P_k(t)=\Pbb(A_k\le t)=F(t-k\Delta),
\qquad
p_k(t)=f(t-k\Delta).
\end{equation}
\subsection{Basic finiteness property}
We first verify that the counting process is well-defined, i.e., that only finitely many arrivals occur in any finite time interval.
\begin{proposition}\label{prop:finiteness}
For every $t\in[0,T]$, the random variable
$
A(t)$
is finite almost surely.
\end{proposition}

\begin{proof}
Let $I_k(t):=\mathbf 1_{\{A_k\in(0,t]\}}$. Then $A(t)=\sum_{k\in\Z} I_k(t)$, and
\[
\mathbb E[A(t)]
=\sum_{k\in\Z}\mathbb P(A_k\in(0,t])
=\sum_{k\in\Z}\big(F(t-k\Delta)-F(-k\Delta)\big).
\]
Since $F$ has density $f\in L^1(\mathbb R)$, Tonelli's theorem yields
\[
\mathbb E[A(t)]
=\int_{\mathbb R} f(u)\sum_{k\in\Z}\mathbf 1_{[-k\Delta,\,t-k\Delta]}(u)\,du.
\]
For each $u\in\mathbb R$ the set of indices $k$ such that $u\in[-k\Delta,\,t-k\Delta]$ has cardinality at most $\lceil t/\Delta\rceil+1$, hence
\[
\mathbb E[A(t)] \le C\|f\|_{L^1} < \infty.
\]
Thus $A(t)$ is integrable and therefore finite almost surely.
\end{proof}

\subsection{Bounded periodicized density assumption}
In order to control the aggregate contribution of all potential arrivals, it is convenient to introduce the following \textit{periodicized density}.
\begin{equation}\label{eq:rho}
\rho(u) := \sum_{k\in\Z} f(u-k\Delta), \qquad u\in[0,T].
\end{equation}
The function $\rho(u)$ represents the total density at time $u$ obtained by summing the contributions of all shifted arrival distributions. 

The following assumption ensures that the total mass of arrivals remains uniformly controlled over the interval $[0,T]$.
\begin{assumption}[Bounded periodicized density]\label{ass:Per}
\begin{equation}\label{eq:Per}
\sup_{u\in[0,T]} \rho(u) < \infty.
\end{equation}
\end{assumption}

Assumption \ref{ass:Per} plays a central role in the analysis. It ensures:
\begin{itemize}
\item uniform bounds on the mass of arrivals over short time intervals, which are crucial for exponential tightness;
\item finiteness of exponential moments of the counting process at fixed times.
\end{itemize}



We first derive a bound on the expected number of arrivals in a time interval.
\begin{proposition}[Increment mass bound]\label{prop:incmass}
Under Assumption \ref{ass:Per}, there exists $C > 0$ such that for all $0\le s<t\le T$,
\begin{equation}\label{eq:incmass}
\sum_{k\in\Z}\Pbb(A_k\in(s,t]) \le C(t-s).
\end{equation}
\end{proposition}

\begin{proof}
The result follows directly from the definition of $\rho$:
\begin{equation}
\sum_{k\in\Z}\Pbb(A_k\in(s,t])=\sum_{k\in\Z}\int_s^t f(u-k\Delta)\,du
=\int_s^t \rho(u)\,du \le (t-s)\sup_{[0,T]}\rho(u).
\end{equation}
Hence, the claim holds with $C :=\sup_{u\in[0,T]}\rho(u)$.
\end{proof}
We next establish exponential integrability of the counting process at fixed times.
Observe that it entails that no Cramér-type moment assumption on the unpunctuality distribution is required.

\begin{proposition}\label{prop:exp-moments}
Under Assumption~\ref{ass:Per}, 
$
\E e^{\lambda A(t)}<\infty,
$ for any $t\in[0,T]$ and $\lambda\in\R$.
\end{proposition}

\begin{proof}
Let $q_k(t):=\Pbb(A_k\in(0,t])$. Since the indicators
$\1_{\{A_k\in(0,t]\}}$ are independent,
\[
\E e^{\lambda A(t)}
=\prod_{k\in\Z}\bigl(1+q_k(t)(e^\lambda-1)\bigr).
\]
Hence, using $\log(1+u)\le u$,
\[
\log\E e^{\lambda A(t)}
\le (e^\lambda-1)\sum_{k\in\Z}q_k(t)
=(e^\lambda-1)\int_0^t\rho(u)\,du,
\]
which is finite by Assumption~\ref{ass:Per}.
\end{proof}

\section{Many-sources scaling and Laplace functionals}\label{sec:log-laplace}

We introduce the many-sources scaling and derive an explicit expression for the log-Laplace functional of the arrival process, which forms the basis of the large deviations analysis.

\subsection{Empirical mean process}

Let $A^{(1)}(\cdot),\ldots,A^{(n)}(\cdot)$ be i.i.d.\ copies of the counting process~\eqref{eq:At}, and define the empirical mean
\begin{equation}\label{eq:empirical}
\bar A^{(n)}(t)
=\frac1n\sum_{i=1}^nA^{(i)}(t),
\qquad t\in[0,T].
\end{equation}
Our goal is to establish a sample-path LDP for $\{\bar A^{(n)}\}_{n\ge1}$.

\subsection{A path identity}

With  $\theta:[0,T]\to\R$ a bounded measurable function, by Proposition~\ref{prop:finiteness},
\[
\int_0^t\theta(s)A(s)\,ds
=\sum_{k\in\Z}\int_0^t\theta(s)\1_{\{0<A_k\le s\}}\,ds.
\]
Since
$
\1_{\{0<A_k\le s\}}
=\1_{\{A_k\in(0,t]\}}\1_{\{s\ge A_k\}},$
we obtain
\begin{equation}\label{eq:keyidentity}
\int_0^t\theta(s)A(s)\,ds
=
\sum_{k\in\Z}
\1_{\{A_k\in(0,t]\}}
\int_{A_k}^t\theta(r)\,dr,
\qquad t\in[0,T].
\end{equation}
Exponentiating both sides yields
\begin{equation}\label{eq:prod}
\exp\!\left(\int_0^t\theta(s)A(s)\,ds\right)
=
\prod_{k\in\Z}
\exp\!\left(
\1_{\{A_k\in(0,t]\}}
\int_{A_k}^t\theta(r)\,dr
\right),
\end{equation}
where the product is almost surely finite.

\subsection{The log-Laplace functional}
We now derive an explicit expression for the log-Laplace functional
\begin{equation}\label{eq:LambdaDef}
\Lambda_t(\theta)
:=
\log\E\exp\!\left(\int_0^t\theta(s)A(s)\,ds\right).
\end{equation}
By \eqref{eq:prod} and the independence of the random variables $\{A_k\}_{k\in\Z}$,
\begin{equation}\label{eq:LambdaSum}
\Lambda_t(\theta)
=
\sum_{k\in\Z}\Lambda_{t,k}(\theta),
\qquad
\Lambda_{t,k}(\theta)
:=
\log\E\exp\!\left(
\1_{\{A_k\in(0,t]\}}
\int_{A_k}^t\theta(r)\,dr
\right).
\end{equation}
Writing
\[
H_\theta(s,t):=\int_s^t\theta(r)\,dr,
\]
and conditioning on $A_k$ gives
\[
\E\exp\!\left(
\1_{\{A_k\in(0,t]\}}H_\theta(A_k,t)
\right)
=
1-P_k(t)+P_k(0)
+\int_0^t e^{H_\theta(s,t)}p_k(s)\,ds.
\]
Hence
\begin{equation}\label{eq:Lambda_k}
\Lambda_{t,k}(\theta)
=
\log\!\left(
1-P_k(t)+P_k(0)
+\int_0^t
e^{\int_s^t\theta(r)\,dr}
\,p_k(s)\,ds
\right).
\end{equation}
\subsection{The scaled log-Laplace functional}
We now compute the log-Laplace functional of the empirical mean process. Using the definition of $\bar A^{(n)}$, and independence of the sources, we have
\begin{align}
\E\exp\!\left(
n\int_0^t \theta(s)\bar A^{(n)}(s)\,ds
\right)
&=
\E\exp\!\left(
\sum_{i=1}^n \int_0^t \theta(s)A^{(i)}(s)\,ds
\right)\nonumber\\
&=
\prod_{i=1}^n
\E\exp\!\left(\int_0^t \theta(s)A^{(i)}(s)\,ds\right)\nonumber\\
&=
\exp\big(n\Lambda_t(\theta)\big).\label{eq:manysourcesLaplace}
\end{align}
Taking logarithms and dividing by $n$, we conclude that the scaled log-Laplace transform of $\bar A^{(n)}$ is given by $\Lambda_t(\theta)$.

\section{Main result}\label{sec:main_theorem}
We are now in a position to state and prove the main result of the paper.
\begin{theorem}\label{thm:main_ldp}
The sequence $\{\bar A^{(n)}\}_{n\ge1}$ satisfies a large deviations principle on $D([0,T])$, equipped with the Skorokhod $J_1$ topology, with speed $n$ and good rate function
\[
I_T(\phi)
=
\sup_{m\ge1}\sup_{0<t_1<\cdots<t_m\le T}
I_{t_1,\dots,t_m}\bigl(\phi(t_1),\dots,\phi(t_m)\bigr),
\]
where, for each $0<t_1<\cdots<t_m\le T$,
\[
I_{t_1,\dots,t_m}(x)
=
\sup_{\alpha\in\mathbb R^m}
\left\{
\sum_{\ell=1}^m \alpha_\ell x_\ell-\Lambda_{t_1,\dots,t_m}(\alpha)
\right\}.
\]
Moreover, the rate function admits the variational representation
\begin{equation}\label{eq:thm_variational_rate}
    I_T(\phi) = \sup_{\theta\in C([0,T]),\gamma\in\R} \left\{ \int_0^T \theta(t) \phi(t)\, dt + \gamma\phi(T) - \Lambda_T(\theta,\gamma)\right\},
\end{equation}
with $\Lambda_T(\theta,\gamma)$ given by
\begin{equation}
    \Lambda_T(\theta,\gamma):=\log \E \exp \left(\int_0^T \theta(t)A(t)\, dt + \gamma A(T)\right).
\end{equation}
\end{theorem}



The proof proceeds in four steps:
\begin{enumerate}
\item establish a finite-dimensional LDP via the G\"artner--Ellis theorem;
\item prove exponential tightness in $D([0,T])$;
\item combine these to obtain a sample-path LDP;
\item derive the dual variational representation of the rate function.
\end{enumerate}

For convenience, we first recall two standard results used to pass from finite-dimensional large deviations to a sample-path LDP.

\begin{theorem}[Theorem~4.1 of \cite{FK06}]\label{thm:FK}
Let $\{X^{(n)}\}$ be a sequence of processes in $D([0,T])$. Suppose that
\begin{enumerate}
\item there exists a dense set $T_0\subset[0,T]$ such that, for every $t\in T_0$, the sequence $\{X^{(n)}(t)\}$ is exponentially tight in $\R$;
\item for every $\varepsilon>0$,
\[
\lim_{\delta\downarrow0}\limsup_{n\to\infty}\frac1n
\log\Pbb\bigl(w'(X^{(n)},\delta,T)>\varepsilon\bigr)
=-\infty.
\]
\end{enumerate}
Then $\{X^{(n)}\}$ is exponentially tight in $D([0,T])$.
\end{theorem}

The next result is the standard criterion combining finite-dimensional LDPs with exponential tightness to obtain a sample-path LDP. It follows, for example, from the Dawson--Gärtner projective limit theorem together with exponential tightness; see \cite[Section~4.6]{DZ98} and \cite[Section~4]{FK06}.

\begin{theorem}[LDP via finite-dimensional LDP and exponential tightness]\label{thm:DG}
Let $\{X^{(n)}\}$ be a sequence of random elements of $D([0,T])$. Suppose that, for every $0<t_1<\cdots<t_m\le T$, the finite-dimensional projections
\[
\bigl(X^{(n)}(t_1),\ldots,X^{(n)}(t_m)\bigr)
\]
satisfy an LDP in $\R^m$ with speed $n$ and good rate function $I_{t_1,\ldots,t_m}$, and that $\{X^{(n)}\}$ is exponentially tight in $D([0,T])$. Then $\{X^{(n)}\}$ satisfies an LDP in $D([0,T])$ with speed $n$ and good rate function
\[
I(\phi)
=
\sup_{m\ge1}\;
\sup_{0<t_1<\cdots<t_m\le T}
I_{t_1,\ldots,t_m}\bigl(\phi(t_1),\ldots,\phi(t_m)\bigr).
\]
\end{theorem}

The next four subsections carry out the four steps above.
\subsection{Finite-dimensional large deviations}\label{sec:finite_LDP}
We first establish an LDP for the finite-dimensional marginals of the empirical mean process. Fix $m\ge1$ and
$
0<t_1<\cdots<t_m\le T,$
and define
\begin{equation}\label{eq:fd-vector}
X^{(n)}
:=
\bigl(\bar A^{(n)}(t_1),\ldots,\bar A^{(n)}(t_m)\bigr)
\in\R^m.
\end{equation}
For $\alpha=(\alpha_1,\ldots,\alpha_m)\in\R^m$, let
\begin{equation}\label{eq:fd-logmgf}
\Lambda_{t_1,\ldots,t_m}(\alpha)
:=
\log\E\exp\!\left(
\sum_{\ell=1}^m\alpha_\ell A(t_\ell)
\right).
\end{equation}
Since the processes $A^{(1)},\ldots,A^{(n)}$ are i.i.d.,
\begin{align}
\frac1n
\log\E\exp\!\left(
n\sum_{\ell=1}^m\alpha_\ell\bar A^{(n)}(t_\ell)
\right)
&=
\frac1n
\log
\prod_{i=1}^n
\E\exp\!\left(
\sum_{\ell=1}^m\alpha_\ell A^{(i)}(t_\ell)
\right)\nonumber\\
&=
\Lambda_{t_1,\ldots,t_m}(\alpha).
\label{eq:fd-scaled-logmgf}
\end{align}
Hence
\begin{equation}\label{eq:fd-logmgf-limit}
\lim_{n\to\infty}
\frac1n
\log\E\exp\!\left(
n\sum_{\ell=1}^m\alpha_\ell\bar A^{(n)}(t_\ell)
\right)
=
\Lambda_{t_1,\ldots,t_m}(\alpha).
\end{equation}
We now derive an explicit expression for $\Lambda_{t_1,\dots,t_m}(\alpha)$. Set $t_0:=0$ and define
\begin{equation}\label{eq:beta-def}
\beta_\ell := \sum_{r=\ell}^m \alpha_r,
\qquad \ell=1,\dots,m.
\end{equation}
Then summation by parts gives
\begin{equation}\label{eq:alpha-beta-identity}
\sum_{\ell=1}^m \alpha_\ell A(t_\ell)
=
\sum_{\ell=1}^m \beta_\ell\big(A(t_\ell)-A(t_{\ell-1})\big).
\end{equation}
Indeed, expanding the right-hand side shows that the coefficient of $A(t_\ell)$ is precisely
$\beta_\ell-\beta_{\ell+1}=\alpha_\ell$, with the convention $\beta_{m+1}:=0$.
Now observe that
\begin{equation}\label{eq:increment-decomposition}
A(t_\ell)-A(t_{\ell-1})
=
\sum_{k\in\Z}\1_{\{A_k\in(t_{\ell-1},t_\ell]\}},
\qquad \ell=1,\dots,m.
\end{equation}
Substituting \eqref{eq:increment-decomposition} into \eqref{eq:alpha-beta-identity}, we obtain
\begin{equation}\label{eq:sum-over-k}
\sum_{\ell=1}^m \alpha_\ell A(t_\ell)
=
\sum_{k\in\Z}
\sum_{\ell=1}^m
\beta_\ell \1_{\{A_k\in(t_{\ell-1},t_\ell]\}}.
\end{equation}
For each fixed $k$, at most one of the events
$
\{A_k\in(t_{\ell-1},t_\ell]\}$, with 
$\ell=1,\dots,m$,
can occur. Define
\begin{equation}\label{eq:pkl-def}
p_{k,\ell}
:=
\Pbb\big(A_k\in(t_{\ell-1},t_\ell]\big)
=
F(t_\ell-k\Delta)-F(t_{\ell-1}-k\Delta),
\qquad \ell=1,\dots,m.
\end{equation}
It follows that, for each $k$,
\begin{align}
\E\exp\!\left(
\sum_{\ell=1}^m
\beta_\ell \1_{\{A_k\in(t_{\ell-1},t_\ell]\}}
\right)
&=
\Big(1-\sum_{\ell=1}^m p_{k,\ell}\Big)
+
\sum_{\ell=1}^m p_{k,\ell}e^{\beta_\ell}.
\label{eq:single-k-fd}
\end{align}
Hence, using \eqref{eq:sum-over-k} and the independence of the random variables $\{A_k\}_{k\in\Z}$,  
\begin{align}
\Lambda_{t_1,\dots,t_m}(\alpha)
&=
\log
\E\exp\!\left(
\sum_{k\in\Z}
\sum_{\ell=1}^m
\beta_\ell \1_{\{A_k\in(t_{\ell-1},t_\ell]\}}
\right)\nonumber=
\sum_{k\in\Z}
\log\!\left(
1 
+
\sum_{\ell=1}^m p_{k,\ell} (e^{\beta_\ell}-1)
\right).
\label{eq:fd-Lambda-explicit}
\end{align}
We next verify the assumptions needed for the Gärtner--Ellis theorem. 
Here it is directly seen that, by \eqref{eq:fd-scaled-logmgf}, the limiting logarithmic moment generating function exists. 

It remains to verify the hypotheses of the G\"artner--Ellis theorem. Since $A$ is nondecreasing,
\[
\sum_{\ell=1}^m\alpha_\ell A(t_\ell)
\le
\Bigl(\sum_{\ell=1}^m\alpha_\ell^+\Bigr)A(t_m),
\]
so
\[
\E\exp\!\left(\sum_{\ell=1}^m\alpha_\ell A(t_\ell)\right)
\le
\E\exp\!\left(
\Bigl(\sum_{\ell=1}^m\alpha_\ell^+\Bigr)A(t_m)
\right),
\]
which is finite by Proposition~\ref{prop:exp-moments}. Hence the effective domain of $\Lambda_{t_1,\ldots,t_m}$ is all of~$\R^m$.

Next,
\[
\Lambda_{t_1,\ldots,t_m}(\alpha)
=
\sum_{k\in\Z}g_k(\alpha),
\]
where
\[
g_k(\alpha)
:=
\log\!\left(
1+\sum_{\ell=1}^mp_{k,\ell}(e^{\beta_\ell}-1)
\right).
\]
If $K\subset\R^m$ is compact, then $|\beta_\ell|\le M$ on $K$ for some $M<\infty$, so
\[
1+\sum_{\ell=1}^mp_{k,\ell}(e^{\beta_\ell}-1)\ge e^{-M}.
\]
Consequently,
\[
|\partial_{\alpha_i}g_k(\alpha)|
\le
e^{2M}\sum_{\ell=1}^mp_{k,\ell},
\qquad
\alpha\in K.
\]
Since
\[
\sum_{k,\ell}p_{k,\ell}
=\E[A(t_m)]<\infty,
\]
the derivative series converges locally uniformly, and $\Lambda_{t_1,\ldots,t_m}$ is differentiable on~$\R^m$.

The G\"artner--Ellis theorem therefore applies, yielding an LDP for $\{X^{(n)}\}$ on $\R^m$ with speed $n$ and good convex rate function
\begin{equation}\label{eq:fd-rate-function}
I_{t_1,\ldots,t_m}(x)
=
\sup_{\alpha\in\R^m}
\left\{
\sum_{\ell=1}^m\alpha_\ell x_\ell
-
\Lambda_{t_1,\ldots,t_m}(\alpha)
\right\},
\qquad
x\in\R^m.
\end{equation}
\subsection{Exponential tightness}

We prove that the sequence $\{\bar A^{(n)}\}_{n\ge1}$ is exponentially tight in $D([0,T])$, equipped with the Skorokhod $J_1$ topology. To this end, we apply \cite[Theorem 4.1]{FK06}, which we stated, for completeness, in Theorem \ref{thm:FK}.

Recall Definition \ref{def:wprime} of the Skorokhod modulus $w'$ and let $T_0:=\mathbb{Q}\cap[0,T]$, which is dense in $[0,T]$. We now prove the two conditions of Theorem \ref{thm:FK}.
 
\medskip

\noindent Proof of (1).
Fix $t\in T_0$. We show that $\{\bar A^{(n)}(t)\}_{n\ge1}$ is exponentially tight in $\R$, that is,
\begin{equation}\label{eq:exp_tight}
\lim_{R\to\infty}\ \limsup_{n\to\infty}\ \frac1n
\log \Pbb(\bar A^{(n)}(t)\ge R)=-\infty.
\end{equation}

For any $R>0$ and $\lambda>0$, Chernoff's bound and independence across sources give
\begin{align}
\Pbb(\bar A^{(n)}(t)\ge R)
&=
\Pbb\Big(\sum_{i=1}^n A^{(i)}(t)\ge nR\Big)
\le
e^{-n\lambda R}\,\E\exp\Big(\lambda\sum_{i=1}^n A^{(i)}(t)\Big) \nonumber\\
&=
e^{-n\lambda R}\,\Big(\E e^{\lambda A(t)}\Big)^n
=
\exp\Big(n\big[-\lambda R+\log\E e^{\lambda A(t)}\big]\Big).\label{eq:marginal-chernoff}
\end{align}
By Proposition \ref{prop:exp-moments}, $\E e^{\lambda A(t)}<\infty$, for every $\lambda>0$. Then,
\[
\limsup_{n\to\infty}\frac1n\log\Pbb(\bar A^{(n)}(t)\ge R)
\le -\lambda R+\log\E e^{\lambda A(t)}.
\]
Letting $R\to\infty$ leads to the exponential tightness condition \eqref{eq:exp_tight}, concluding the proof of (1).

\medskip

\noindent Proof of (2). Fix $\varepsilon>0$. We prove the Feng--Kurtz condition (d):
\begin{equation}\label{eq:FKd-goal}
\lim_{\delta\downarrow0}\ \limsup_{n\to\infty}\ \frac1n
\log\Pbb\big\{w'(\bar A^{(n)},\delta,T)>\varepsilon\big\}=-\infty.
\end{equation}

We proceed in four steps:
\begin{enumerate}[(a)]
    \item We reduce the Skorokhod modulus $w'$ to increments, using the monotonicity of the paths.
    \item We derive a bound on the log-moment generating function of increments of the process.
    \item We obtain exponential tail bounds for increments of the empirical mean process via Chernoff's inequality. 
    \item We control the supremum over all short time intervals by a discretization argument.
\end{enumerate}

\noindent {\it Step 2(a).}
Recall that each $\bar A^{(n)}$ is nondecreasing and càdlàg. For a nondecreasing path $\phi$, the oscillation on any interval $[a,b]\subset[0,T]$ is simply
\[
\sup_{s,t\in[a,b]}|\phi(t)-\phi(s)|
=
\phi(b)-\phi(a).
\]

Now choose an admissible partition
$\Pi=\{0=t_0<t_1<\cdots<t_m=T\}$
such that
$
\delta \le t_i-t_{i-1}\le 2\delta$ for all 
$i=1,\dots,m.$
Using this partition in the definition of $w'$, we obtain
\[
w'(\phi,\delta,T)
\le
\max_i \sup_{s,t\in[t_{i-1},t_i]}|\phi(t)-\phi(s)|
=
\max_i \big(\phi(t_i)-\phi(t_{i-1})\big).
\]
Since each interval $[t_{i-1},t_i]$ has length at most $2\delta$, it follows that
\[
\phi(t_i)-\phi(t_{i-1})
\le
\sup_{\substack{0\le s<t\le T\\ t-s\le 2\delta}}
\big(\phi(t)-\phi(s)\big).
\]
Taking the maximum over $i$ yields
\begin{equation}\label{eq:wprime-inc-bound}
w'(\phi,\delta,T)\ \le\ \sup_{\substack{0\le s<t\le T\\ t-s\le 2\delta}}\big(\phi(t)-\phi(s)\big),
\qquad \phi\in D([0,T]).
\end{equation}

\noindent {\it Step 2(b)}.
Fix $0\le s<t\le T$. Since
\[
A(t)-A(s)=\sum_{k\in\Z}\1_{\{A_k\in(s,t]\}},
\]
and the indicators are independent across $k$, for $\lambda>0$,
\begin{align}
\log\E e^{\lambda(A(t)-A(s))}
&=
\sum_{k\in\Z}\log\Big(1+\Pbb(A_k\in(s,t])(e^\lambda-1)\Big)\nonumber\\
&\le (e^\lambda-1)\sum_{k\in\Z}\Pbb(A_k\in(s,t])\le (e^\lambda-1)\,C\,(t-s),\label{eq:one-stream-inc-mgf}
\end{align}
where we used $\log(1+u)\le u$ and Proposition \ref{prop:incmass}.

\noindent {\it Step 2(c).}
For $\eta>0$ and $\lambda>0$, Chernoff's bound and independence across sources yield
\begin{align}
\Pbb\big(\bar A^{(n)}(t)-\bar A^{(n)}(s)\ge \eta\big)
&=
\Pbb\Big(\sum_{i=1}^n (A^{(i)}(t)-A^{(i)}(s))\ge n\eta\Big)\nonumber\\
&\le
\exp(-n\lambda\eta)\,
\Big(\E e^{\lambda(A(t)-A(s))}\Big)^n.\label{eq:emp-inc-chernoff}
\end{align}
Combining \eqref{eq:emp-inc-chernoff} with \eqref{eq:one-stream-inc-mgf} gives
\begin{equation}\label{eq:emp-inc-exp-bound}
\frac1n\log \Pbb\big(\bar A^{(n)}(t)-\bar A^{(n)}(s)\ge \eta\big)
\le
-\lambda\eta+(e^\lambda-1)\,C\,(t-s).
\end{equation}

\noindent {\it Step 2(d).}
Fix $\delta>0$ and define the grid $u_j:=j\delta$ for $j=0,1,\dots,M$, where
\[
M:=\lceil T/\delta\rceil,
\]
and set $u_M:=T$.
Consider any $0\le s<t\le T$ such that $t-s\le 2\delta$. Choose $j\in\{0,\dots,M-1\}$ such that
$
s\in[u_j,u_{j+1}],
$
and define
$
r(j):= \min\{j+3, M\}.
$
Since $t-s\le 2\delta$ and $s\le u_{j+1}$, it follows that
\[
t\le s+2\delta\le u_{j+1}+2\delta. 
\]
If $j+3< M$, then $u_{j+1} + 2\delta = u_{j+3}$, and hence
$$
t\le u_{j+3} = u_{r(j)}.
$$
If $j+3 \ge M$, then $r(j) = M$ and $u_{r(j)} = T$, so trivially $t\le T = u_{r(j)}$.
Thus, in all cases
\[
u_j\le s<t\le u_{r(j)},
\]
(see Figure \ref{fig:grid_discretization}).
Since $\bar A^{(n)}$ is nondecreasing, we obtain
\[
\bar A^{(n)}(t)-\bar A^{(n)}(s)
\le
\bar A^{(n)}(u_{r(j)})-\bar A^{(n)}(u_j).
\]


Fix $\eta>0$. By the union bound,
\begin{align}
\Pbb\left(
\sup_{\substack{0\le s<t\le T\\ t-s\le 2\delta}}
(\bar A^{(n)}(t)-\bar A^{(n)}(s))\ge \eta
\right)
&\le
\sum_{j=0}^{M-1}
\Pbb\big(
\bar A^{(n)}(u_{r(j)})-\bar A^{(n)}(u_j)\ge \eta
\big).
\label{eq:union-short-intervals}
\end{align}
Since $u_{r(j)}-u_j\le 3\delta$, the bound \eqref{eq:emp-inc-exp-bound} yields
\[
\Pbb\big(
\bar A^{(n)}(u_{r(j)})-\bar A^{(n)}(u_j)\ge \eta
\big)
\le
\exp\Big(n\big[-\lambda\eta+(e^\lambda-1)C(3\delta)\big]\Big).
\]
Substituting into \eqref{eq:union-short-intervals}, we obtain
\begin{equation}\label{eq:union-bound}
\Pbb\left(
\sup_{\substack{0\le s<t\le T\\ t-s\le 2\delta}}
(\bar A^{(n)}(t)-\bar A^{(n)}(s))\ge \eta
\right)
\le
M\cdot
\exp\Big(n\big[-\lambda\eta+(e^\lambda-1)C(3\delta)\big]\Big).
\end{equation}
Taking $\frac1n\log$ and $\limsup_{n\to\infty}$, and using that $M$ does not depend on $n$, gives
\[
\limsup_{n\to\infty}\frac1n\log\Pbb\left(
\sup_{\substack{0\le s<t\le T\\ t-s\le 2\delta}}
(\bar A^{(n)}(t)-\bar A^{(n)}(s))\ge \eta
\right)
\le
-\lambda\eta+(e^\lambda-1)C(3\delta).
\]
Letting $\delta\downarrow0$ yields
\[
\lim_{\delta\downarrow0}\ \limsup_{n\to\infty}\frac1n\log\Pbb\left(
\sup_{\substack{0\le s<t\le T\\ t-s\le 2\delta}}
(\bar A^{(n)}(t)-\bar A^{(n)}(s))\ge \eta
\right)
\le -\lambda\eta.
\]
Since $\lambda>0$ is arbitrary, we conclude
\[
\lim_{\delta\downarrow0}\ \limsup_{n\to\infty}\frac1n\log\Pbb\left(
\sup_{\substack{0\le s<t\le T\\ t-s\le 2\delta}}
(\bar A^{(n)}(t)-\bar A^{(n)}(s))\ge \eta
\right)=-\infty.
\]
Combining this with \eqref{eq:wprime-inc-bound} yields the Feng--Kurtz condition \eqref{eq:FKd-goal}.

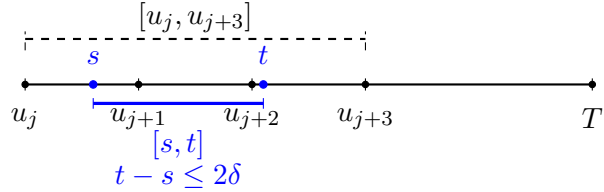
\begin{figure}[h]
\centering
\begin{tikzpicture}[x=1.5cm,y=1cm]

\draw[thick] (0,0) -- (5,0);

\foreach \x/\lab in {0/{u_j},1/{u_{j+1}},2/{u_{j+2}},3/{u_{j+3}},5/{T}}
{
    \filldraw (\x,0) circle (1.2pt);
    \draw (\x,0.08) -- (\x,-0.08);
    \node[below=5pt] at (\x,0) {$\lab$};
}

\draw[dashed, thick] (0,0.60) -- (3,0.60);
\draw[dashed] (0,0.45) -- (0,0.70);
\draw[dashed] (3,0.45) -- (3,0.70);
\node[above] at (1.5,0.55) {$[u_j,u_{j+3}]$};

\filldraw[blue] (0.6,0) circle (1.4pt);
\filldraw[blue] (2.1,0) circle (1.4pt);

\node[above=4pt, blue] at (0.6,0) {$s$};
\node[above=4pt, blue] at (2.1,0) {$t$};

\draw[blue, very thick] (0.6,-0.25) -- (2.1,-0.25);
\draw[blue] (0.6,-0.18) -- (0.6,-0.32);
\draw[blue] (2.1,-0.18) -- (2.1,-0.32);
\node[below=6pt, blue] at (1.35,-0.25) {$[s,t]$};
\node [below=6pt, blue] at (1.35,-0.70) {$t-s\le 2\delta$};

\end{tikzpicture}
\caption{Illustration of the discretization argument in Step 2(d), away from the
right endpoint. If \(s\in[u_j,u_{j+1}]\) and \(t-s\le 2\delta\), then
\(t\le u_{r(j)}\), where \(r(j)=\min\{j+3,M\}\). In the displayed case
\(r(j)=j+3\), so \([s,t]\subset[u_j,u_{j+3}]\).}
\label{fig:grid_discretization}
\end{figure}

\subsection{Proof of the sample-path LDP}

By Theorem~\ref{thm:DG}, the finite-dimensional LDP obtained above and the
exponential tightness proved in the previous subsection imply that
\(\{\bar A^{(n)}\}\) satisfies an LDP on \(D([0,T])\) with speed \(n\) and good
rate function
\begin{equation}
\label{eq:rate_function}
I_T(\phi)
=
\sup_{m\ge1}
\sup_{0<t_1<\cdots<t_m\le T}
I_{t_1,\dots,t_m}
\bigl(\phi(t_1),\dots,\phi(t_m)\bigr).
\end{equation}

\subsection{Variational representation of the rate function}\label{subsec:rate_function}

We now prove that the rate function $I_T$ in \eqref{eq:rate_function} has the variational form \eqref{eq:thm_variational_rate}. We split the proof in 4 steps.

\medskip

\noindent {\it Step 1: Continuity of the test functional.}
For $\theta\in C([0,T])$ and $\gamma\in\R$, define
\[
F_{\theta,\gamma}(\phi):=\int_0^T \theta(t)\phi(t)\,dt + \gamma \phi(T),
\qquad \phi\in D([0,T]).
\]
We also define the corresponding log-Laplace functional
\begin{equation}
    \Lambda_T(\theta,\gamma) := \log \E \exp\left(\int_0^T \theta(t) A(t)\, dt + \gamma A(T)\right).
\end{equation}
Notice that $\Lambda_T(\theta,0) = \Lambda_T(\theta)$.
\begin{lemma}\label{lem:Ftheta-cont}
For every $\theta\in C([0,T])$ and $\gamma\in\R$, the map $F_{\theta,\gamma}$ is continuous on $D([0,T])$ endowed with the Skorokhod $J_1$ topology.
\end{lemma}

\begin{proof}
If $\phi_n \to \phi$ in the Skorokhod $J_1$ topology, then $\phi_n(t)\to\phi(t)$ at every continuity point of $\phi$ (see, e.g., \cite[Chapter 3]{Bill99}). Since $\phi$ is càdlàg, this convergence holds for Lebesgue-a.e.\ $t\in[0,T]$. The map $\phi\mapsto \phi(T)$ is also continuous in the $J_1$ topology, see \cite[Section~3.3]{Kern2024}. The result then follows by dominated convergence.
\end{proof}

\medskip

\noindent {\it Step 2: Laplace transform limit.}
The same argument used for \eqref{eq:manysourcesLaplace} gives that for every $\theta\in C([0,T])$ and $\gamma\in\R$,
\begin{equation}\label{eq:scaled-laplace-limit}
\lim_{n\to\infty}\frac1n
\log \E \exp\!\left(
nF_{\theta,\gamma}(\bar A^{(n)})
\right)
=
\Lambda_T(\theta,\gamma).
\end{equation}

\medskip

\noindent {\it Step 3: Exponential integrability and Varadhan's lemma.}
We apply an extended form of Varadhan's lemma for continuous functionals with sufficient exponential integrability; see, e.g., \cite[Section 4.3]{DZ98}.

Let $p>1$. Since each path $\bar A^{(n)}$ is nonnegative and nondecreasing,
\[
F_{\theta,\gamma}(\bar A^{(n)})
\le
\int_0^T |\theta(t)|\,\bar A^{(n)}(t)\,dt + |\gamma|\bar A^{(n)}(T)
\le
\left(T\|\theta\|_\infty + |\gamma|\right)\bar A^{(n)}(T).
\]
Therefore,
\begin{align}
\E \exp\!\big(p\, nF_{\theta,\gamma}(\bar A^{(n)})\big)
&\le
\E \exp\!\big(p\, n\left(T\|\theta\|_\infty+|\gamma|\right)\bar A^{(n)}(T)\big)\nonumber\\
&=
\left(\E e^{p\left( T\|\theta\|_\infty +|\gamma|\right) A(T)}\right)^n.
\end{align}
By Proposition~\ref{prop:exp-moments}, the right-hand side is finite. Hence,
\[
\limsup_{n\to\infty}\frac1n
\log \E \exp\!\big(p nF_{\theta,\gamma}(\bar A^{(n)})\big)
<\infty.
\]
Thus, the exponential integrability condition required for the extended Varadhan lemma is satisfied.
Combining this with \eqref{eq:scaled-laplace-limit}, we obtain
\begin{equation} 
\Lambda_T(\theta,\gamma)
=
\sup_{\phi\in D([0,T])}
\left\{ F_{\theta,\gamma}(\phi) -I_T(\phi)
\right\}=
\notag
    \sup_{\phi\in D([0,T])}\left\{\int_0^T \theta(t)\phi(t)\, dt + \gamma\phi(T) - I_T(\phi)\right\},
\end{equation}
from which it follows that, for every $\phi\in D([0,T])$, $\theta\in C([0,T])$ and $\gamma\in\mathbb R$, 
\begin{equation}
    I_T(\phi)\ge \int_0^T\theta(t)\phi(t)\, dt + \gamma\phi(T) - \Lambda_T(\theta,\gamma).
\end{equation}
Taking the supremum over $\theta$ and $\gamma$ we obtain 
\begin{equation}
    \label{eq:first-ineq-variational}
    I_T(\phi) \ge \sup_{\theta\in C([0,T]), \gamma\in\R}\left\{ \int_0^T\theta(t)\phi(t)\, dt + \gamma\phi(T) - \Lambda_T(\theta,\gamma)\right\}.
\end{equation}
\noindent {\it Step 4: Converse inequality and variational form.}
Define
\[
\mathcal V_T(\phi)
:=
\sup_{\theta\in C([0,T]),\gamma\in\R}
\left\{
\int_0^T \theta(t)\phi(t)\,dt + \gamma\phi(T)
-
\Lambda_T(\theta,\gamma)
\right\}.
\]
We prove that $\mathcal V_T(\phi) \ge I_T(\phi)$. Fix $m\ge1$, $0<t_1<\cdots<t_m\le T$, and $\alpha\in\R^m$. We claim that 
\begin{equation}\label{eq:claim_approximation}
\mathcal V_T(\phi) \ge \sum_{\ell=1}^m \alpha_\ell \phi(t_\ell) -\Lambda_{t_1,\dots,t_m}(\alpha).
\end{equation}
Taking the supremum over $\alpha\in\R^m$ and over all partitions, we obtain $\mathcal V_T(\phi)\ge I_T(\phi)$.

\medskip

Combining both inequalities, we conclude that
$  I_T(\phi) = \mathcal{V}_T(\phi),$ 
which is exactly \eqref{eq:thm_variational_rate}.
The proof of \eqref{eq:claim_approximation} consists in an approximation argument and it is given in
Appendix~\ref{app:duality}.

\section{Examples}\label{sec:examples}
In this section we show our single assumption, Assumption \ref{ass:Per}, for two examples of the random variables $T_k$, namely, the \textit{Laplace} and \textit{Double Pareto unpunctuality}.

\subsection{Laplace unpunctuality}

Let $T_k$ have the \emph{Laplace} density
\[
f(x)=\frac{1}{2b}e^{-|x-\mu|/b},\qquad b>0.
\]
Then, for $u\in[0,T]$,
\[
\rho(u)
=\sum_{k\in\Z}\frac{1}{2b}e^{-|u-k\Delta-\mu|/b}.
\]
Since $M:=\sup_{u\in[0,T]}|u-\mu|<\infty$, we have
$
|u-k\Delta-\mu|\ge |k|\Delta-M,
$
and hence
\[
\rho(u)
\le
\frac{e^{M/b}}{2b}
\sum_{k\in\Z}e^{-|k|\Delta/b}
<\infty,
\]
because the series is geometric. Therefore, Assumption~\ref{ass:Per} holds.

\subsection{Double Pareto unpunctuality}
Let $T_k$ have the \emph{symmetric double Pareto} density satisfying, for some $\alpha>0$ and $C_0<\infty$,
\[
0\le f(x)\le C_0(1+|x|)^{-(\alpha+1)}, \qquad x\in\R.
\]
Then, for any $u\in[0,T]$,
\[
\rho(u)
=\sum_{k\in\Z}f(u-k\Delta)
\le
C_0\sum_{k\in\Z}(1+|u-k\Delta|)^{-(\alpha+1)}.
\]
Since $u\in[0,T]$, we have $1+|k|\Delta\le (1+T)(1+|u-k\Delta|)$, and therefore
\[
\rho(u)
\le
C_0(1+T)^{\alpha+1}
\sum_{k\in\Z}(1+|k|\Delta)^{-(\alpha+1)}.
\]
The series is finite because $\alpha>0$, so Assumption~\ref{ass:Per} holds.

\section{Workload large deviations}\label{sec:workload}
In this section, we derive a sample-path large deviations principle for the workload process.
Let $S^{(n)}(t)$ denote the cumulative service capacity up to time $t$, and define the scaled service process
\[
\bar{S}^{(n)}(t) := \frac{1}{n} S^{(n)}(t).
\]
The net-input process and workload process are defined as
\begin{align*}
    B^{(n)}(t) &:= \bar{A}^{(n)}(t) - \bar{S}^{(n)}(t),\\
    Q^{(n)}(t) &:= \sup_{0\le s \le t} \big( B^{(n)}(t) - B^{(n)}(s) \big).
\end{align*}
Our goal is to establish a sample-path LDP for $\{Q^{(n)}\}$ in $D([0,T])$.
\subsection{Constant service rate}

Assume first that the service rate is constant $c>0$. Then
\[
S^{(n)}(t) = nct,
\qquad
\bar{S}^{(n)}(t) = ct.
\]

\subsubsection*{(1) LDP for the net-input process}

Define the map $\Phi_c:D([0,T])\to D([0,T])$ by
\begin{equation}
\Phi_c(\phi)(t) = \phi(t) - ct.
\end{equation}

Since subtraction of a continuous function preserves the Skorokhod $J_1$ topology,
$\Phi_c$ is continuous (see, e.g., \cite[Chapter 3]{EK86}).
Moreover,
\[
B^{(n)} = \Phi_c(\bar A^{(n)}).
\]
By the contraction principle,
$\{B^{(n)}\}$ satisfies an LDP with good rate function
\begin{equation}
J_T(x) = I_T(\phi_x),
\end{equation}
where the path $\phi_x \in D([0,T])$ is defined by
$$\phi_x(t)  = x(t) + ct \qquad t\in [0,T].$$

\subsubsection*{(2) LDP for the workload process}

Define $\Psi:D([0,T])\to D([0,T])$ by
\begin{equation}\label{eq:Psi}
\Psi(\phi)(t) = \sup_{0\le s \le t} \big(\phi(t) - \phi(s)\big).
\end{equation}

The map $\Psi$ is continuous in the Skorokhod $J_1$ topology
(see \cite[Theorem 13.5.1]{Whitt2002}).

Since
\[
Q^{(n)} = \Psi(B^{(n)}),
\]
another application of the contraction principle yields an LDP for $\{Q^{(n)}\}$ with rate function
\begin{equation}
K_T(q) = \inf \{ J_T(x) : x\in D([0,T]),\ q=\Psi(x)\}.
\end{equation}

\subsubsection*{(3) Tilted arrival path associated with overflow}
We now identify a candidate path describing the most likely way in which a workload overflow occurs. Fix $a>0$. We seek a path that realizes the workload overflow event
$
Q^{(n)}(T)\ge a
$
with minimal large-deviation cost.

For a constant service rate $c$, overflow at time $t\in(0,T]$ requires
$
\phi(t)-ct=a.$
Let
\[
I_A[0,t](x) := \inf\{\,I_T(\phi):\phi(t)=x \,\}.
\]
By the finite-dimensional LDP from Section \ref{sec:finite_LDP},  $\bar A^{(n)}(t)$ satisfies an LDP with rate function
\[
I_A[0,t](x) = \sup_{\theta\in\mathbb R} \Bigl\{ \theta x - \Lambda_t(\theta) \Bigr\},
\]
where
$
\Lambda_t(\theta) := \log \E e^{\theta A(t)}.
$
Hence the minimal cost of producing overflow at time $t$ equals
\[
\inf_{\phi:\,\phi(t)\ge a+ct} I_T(\phi) = \inf_{x\ge a+ct} I_A[0,t](x).
\]
Since $I_A[0,t]$ is convex and attains its minimum at the typical value of $A(t)$, the infimum is achieved at the boundary,
\[
\inf_{\phi:\,\phi(t)\ge a+ct} I_T(\phi) = I_A[0,t](a+ct).
\]
Optimizing over all possible overflow times yields
\begin{equation}
\inf_{0<t\le T} I_A[0,t](a+ct).
\label{eq:overflow_decay_rate}
\end{equation}
Let $x:= a+ ct$ and let $\theta_t$ be such that
$
\Lambda_t'(\theta_t)=x,$
or, equivalently,
\[
\theta_t = \arg\max_{\theta\in\mathbb R} \Bigl\{ \theta x- \Lambda_t(\theta) \Bigr\}.
\]
The parameter $\theta_t$ is therefore the exponential tilting parameter associated with the rare event $\{A(t)\approx x\}$.
Define the exponentially tilted measure
\[
\frac{d\mathbb P^{\theta_t}}{d\mathbb P} = \exp\Bigl(\theta_tA(t)- \Lambda_t(\theta_t)\Bigr) = \frac{e^{\theta_t A(t)}}{\E[e^{\theta_t A(t)}]},
\]
and introduce the path
\begin{equation}
\phi^*_t(u) := \E^{\theta_t}[A(u)] = \frac{\E\left[A(u)e^{\theta_tA(t)}\right]}{\E e^{\theta_tA(t)}},\qquad u\in[0,T].
\label{eq:phi_star_definition}
\end{equation}
By construction,
$
\phi^*_t(t) = \E^{\theta_t}[A(t)] = \Lambda_t'(\theta_t) = x = a+ct.$
Moreover, since
\[
A(t)=\sum_{j\in\mathbb Z} I_j(t), \qquad I_j(t) = \mathbf{1}_{\{A_j \in (0,t]\}},
\]
we have
\[
e^{\theta_t A(t)} = \prod_{j\in\mathbb Z} e^{\theta_t I_j(t)}.
\]
Then, using the independence of the arrivals, 
\[
\mathbb E\!\left[I_k(u)e^{\theta_t A(t)}\right] = \mathbb E\!\left[I_k(u)e^{\theta_t I_k(t)}\right] \prod_{j\neq k} \mathbb E\!\left[e^{\theta_t I_j(t)}\right],
\]
and
\[
\mathbb E[e^{\theta_t A(t)}]=\prod_{j\in\mathbb Z}\mathbb E\!\left[e^{\theta_t I_j(t)}\right].
\]
Hence, 
\[
\phi^*_t(u) = \sum_{k\in\Z} \frac{\E \left[I_k(u) e^{\theta_t I_k(t)}\right]}{\E\left[e^{\theta_t I_k(t)}\right]}.
\]
Recall the definition 
$
q_k(t) :=  \P (0< A_k \le t). 
$
It follows that 
\[
\E\left[e^{\theta_t I_k(t)}\right] = 1+q_k(t)\bigl(e^{\theta_t}-1\bigr).
\]
We now discuss two cases. 

\medskip

\noindent Case 1: \(u\le t\).
Since $\{0< A_k \le u\} \subset \{ 0< A_k \le t\}$, we have
that $I_k(u) e^{\theta_t I_k(t)}$ and $e^{\theta_t} I_k(u)$ coincide, so that 
\[
\E\left[ I_k(u) e^{\theta_t I_k(t)} \right] = e^{\theta_t} \P (0 < A_k \le u). 
\]
It follows that
\[
\phi_t^*(u) = \sum_{k\in\Z} \frac{e^{\theta_t}q_k(u)}{1+ q_k(t) (e^{\theta_t} - 1)}, \qquad u\le t. 
\]

\medskip

\noindent Case 2: $u > t$.
We decompose $\{0< A_k \le u\} = \{ 0 < A_k \le t\} \cup \{ t < A_k \le u\}$. On the first event $I_k(t) = 1$, while on the second event $I_k(t) = 0$. Hence, 
\[
I_k(u) e^{\theta_t I_k(t)} = e^{\theta_t} \mathbf{1}_{\{0<A_k \le t\}} + \mathbf{1}_{\{t< A_k \le u\}}. 
\]
Consequently, 
\[
\phi_t^*(u) = \sum_{k\in\Z}\frac{e^{\theta_t} q_k(t) + \P(t< A_k \le u)}{1 + q_k(t) (e^{\theta_t} -1)}, \qquad u>t. 
\]
The following proposition identifies the large-deviation cost of the path
$\phi_t^*$.
\hypertarget{prop-overflowpath-target}{}
\begin{proposition}\label{prop:ldp_cost_overflowpath}
For every $t\in(0,T]$, the path $\phi_t^*$ defined by
\eqref{eq:phi_star_definition} satisfies
\[
I_T(\phi_t^*) = I_A[0,t](\phi_t^*(t)).
\]
\end{proposition}
\begin{proof}
See Appendix~\ref{app:overflow}.
\end{proof}

In particular, for $t^* := \arg\min_{0 < t \leq T} I_A[0,t](a+ct)$,
\[
I_T(\phi_{t^*}^*) = I_A[0,t^*](a+ct^*) = \inf_{0<t\le T} I_A[0,t](a+ct).
\]
Therefore, the path
$
\phi^* := \phi_{t^*}^*
$
has the same cost as the optimal one-dimensional overflow event \eqref{eq:overflow_decay_rate}.

\begin{remark}
The associated workload path
$
\beta^*(u)=\phi^*(u)-cu
$
builds up until
$
\beta^*(t^*)=a.
$
Thus, in the deterministic-service case, the rare event is realized by an atypical
accumulation of arrivals over the interval \([0,t^*]\), with total mass
\[
\phi^*(t^*)=a+ct^*.
\]
\end{remark}

Figure~\ref{fig:tilted-overflow-path} illustrates the tilted arrival path. The simulated conditional mean, obtained by conditioning on overflow occurring close to $t^*$, agrees well with the theoretical tilted path. This supports the interpretation of $\phi^*_{t^*}$ as the typical arrival trajectory leading to a rare overflow event.

\begin{figure}[h]
    \centering
    \includegraphics[width=0.8\linewidth]{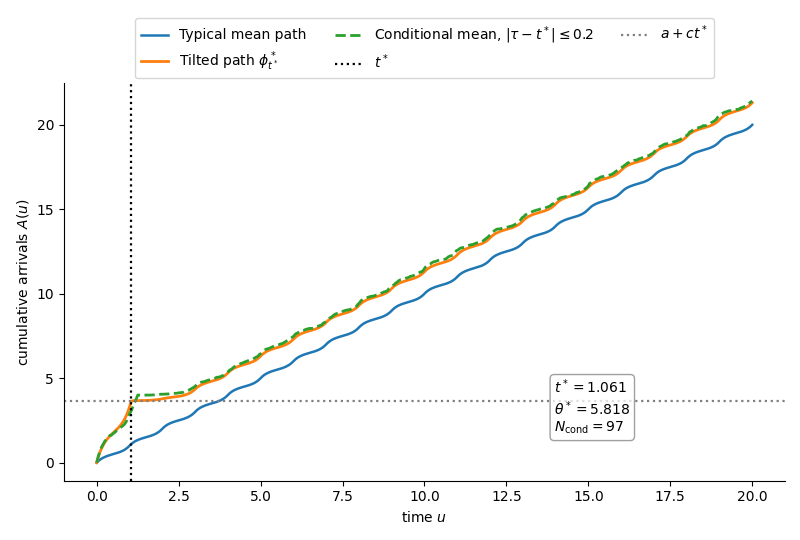}
    \caption{Numerical illustration of the deterministic-service overflow path. The typical mean path $\E[A(u)]$ is compared with the theoretical tilted path $\phi^*_{t^*}$ and with the empirical mean of simulated paths conditioned on overflow occurring within the window $|\tau - t^*| \le 0.2$. The dotted vertical line indicates the optimal overflow time $t^*$, and the dotted horizontal line indicates the required arrival level $a + ct^*$. $N_{\mathrm{cond}}=97$ means that 97 out of $10^6$ simulated paths satisfied the conditioning event. Parameters: $\Delta = 1$, $b=0.25$, $T=20$, $c=1.1$, $a=2.5$.}
    \label{fig:tilted-overflow-path}
\end{figure}
\subsection{Random service times}

Let $(S_i)_{i\ge1}$ be i.i.d.\ positive random variables with log-mgf
$
\Lambda_S(\theta) := \log \E e^{\theta S_1}.
$
Define
\[
V_n := \sum_{i=1}^n S_i,
\qquad
S(t) := \sup\{n\ge0: V_n \le t\}.
\]
The rescaled service process is
\[
\bar S^{(n)}(t) := \frac{1}{n} S(nt).
\]

\subsubsection*{(1) LDP for the service process}

Define the scaled partial-sum process
\[
\bar V^{(n)}(u) := \frac{1}{n} V_{\lfloor nu\rfloor}, 
\]
for $u \in [0,U]$.

By Mogulskii’s theorem (see \cite[Theorem 5.1.2]{DZ98}),
$\{\bar V^{(n)}\}$ satisfies an LDP with rate function
\[
I_V(\phi)
=
\begin{cases}
\int_0^U \Lambda_S^*(\phi'(u))\,du,
& \phi\in AC,\ \phi(0)=0,\\
+\infty & \text{otherwise}.
\end{cases}
\]
With the mapping
$
\Gamma(\phi)(t) := \sup\{u: \phi(u)\le t\}$, we have
that $\bar S^{(n)} = \Gamma(\bar V^{(n)})$.

On the set of strictly increasing absolutely continuous paths,
$\Gamma$ is continuous (see \cite[Section 13]{Whitt2002}), and the contraction principle applies.

Hence $\{\bar S^{(n)}\}$ satisfies an LDP with rate function
\[
L_T(\psi)
=
\begin{cases}
\int_0^T \psi'(t)\,\Lambda_S^*\!\left(\frac{1}{\psi'(t)}\right)\,dt,
& \psi\in AC,\ \psi(0)=0,\ \psi \text{ nondecreasing},\\
+\infty & \text{otherwise}.
\end{cases}
\]

\subsubsection*{(2) LDP for the net-input process}

Assume $\bar A^{(n)}$ and $\bar S^{(n)}$ are independent.
Then we have that
$
\smash{(\bar A^{(n)}, \bar S^{(n)})}
$
satisfies an LDP with rate function
\[
(\phi,\psi)\mapsto I_T(\phi)+L_T(\psi).
\]
Define
$
\Phi(\phi,\psi) := \phi - \psi.$
Although $\Phi$ is not continuous on all of $D\times D$ under $J_1$,
it is continuous on the effective domain where $\psi$ is absolutely continuous.
Therefore, the contraction principle applies.
Hence $\{B^{(n)}\}$ satisfies an LDP with rate
\[
J_T(x)
=
\inf\{I_T(\phi)+L_T(\psi): x=\phi-\psi\}.
\]

\subsubsection*{(3) LDP for the workload process}

Recall the definition of $\Psi$ in \eqref{eq:Psi}. Since $Q^{(n)}=\Psi(B^{(n)})$, another application of the contraction principle yields
\[
K_T(q)
=
\inf\{J_T(x): q=\Psi(x)\}.
\]
Equivalently,
\[
K_T(q)
=
\inf\{I_T(\phi)+L_T(\psi): q=\Psi(\phi-\psi)\}.
\]

\begin{remark}
The case of constant service can be recovered as a degenerate instance of the random service model. Indeed, suppose that the service times are deterministic, i.e.\ $S_i \equiv 1/c$ for some $c>0$. Then the log-moment generating function is
\[
\Lambda_S(\theta) = \log \E e^{\theta S_1} = \frac{\theta}{c}.
\]
Its Legendre transform $\Lambda_S^*(x)$ is $0$ if $x = 1/c$ and $\infty$ otherwise. 
Substituting this into the rate function $L_T$, we obtain
\[
L_T(\psi)
=
\int_0^T \psi'(t)\,\Lambda_S^*\!\left(\frac{1}{\psi'(t)}\right)\,dt.
\]
Hence, this quantity is finite if and only if $\psi'(t)=c$ almost everywhere. Therefore,
\[
L_T(\psi)
=
\begin{cases}
0, & \text{if } \psi(t)=ct \text{ for all } t\in[0,T],\\
+\infty, & \text{otherwise}.
\end{cases}
\]
Hence the only admissible service path is the deterministic trajectory $t \mapsto ct$, and the random service model reduces to the constant service case.
\end{remark}

\begin{remark}
The rate function $K_T(q)$ naturally characterizes the most likely path to overflow in the random-service setting, analogous to the deterministic-service case. Specifically, the decay rate of ${\mathbb P}(Q^{(n)}(t)\ge a)$ is found by minimizing the joint cost $I_T(\phi)+L_T(\psi)$ over all arrival and service paths satisfying
$
\smash{\sup_{0\le s\le t}\bigl[\phi(t)-\phi(s)-\psi(t)+\psi(s)\bigr]\ge a.}
$
Unlike the deterministic-service case, a rare overflow may result from both an atypically large arrival path and an atypically small service path, so the optimal trajectory is the joint minimizer over $(\phi,\psi)$. We do not characterize this minimizer explicitly.
\end{remark}

\appendix 
\section{Approximation of finite-dimensional duals}\label{app:duality}
In this appendix, we prove the approximation argument used in
Step~4 of Section \ref{subsec:rate_function}. The idea is to approximate the values $\phi(t_\ell)$ at non-terminal points by integrals of the form $\smash{\int_0^T \theta_n(t)\phi(t)\,dt}$, while the terminal value, when present, is represented by the term $\gamma \phi(T)$.

Fix $m\ge1$, $0<t_1< \dots <t_m\le T$, and $\alpha\in\mathbb{R}^m$.  We prove that 
\begin{equation}
    \mathcal{V}_T(\phi) \ge \sum_{\ell=1}^m \alpha_\ell \phi(t_\ell)- \Lambda_{t_1,\dots,t_m}(\alpha).
\end{equation}

We distinguish two cases. 

\medskip

\noindent {\it Case 1: $t_m<T$.}
Let $\gamma = 0$. Let $\eta\in C([0,1])$ be nonnegative with $\int_0^1 \eta(u)\,du=1$, and $\eta(0)=\eta(1)=0$. For $\ell=1,\dots,m$, and $n$ sufficiently large,  define
\[
\eta_{n,\ell}(t)
:=
n\,\eta\big(n(t-t_\ell)\big)\mathbf{1}_{[t_\ell,t_\ell+1/n]}(t),
\]
and set
\[
\theta_n(t):=\sum_{\ell=1}^m \alpha_\ell \eta_{n,\ell}(t).
\]
Then $\theta_n\in C([0,T])$ and $\int_0^T \eta_{n,\ell}(t)\,dt=1$.
Moreover, since $\phi$ is càdlàg and hence right-continuous, we have
\[
\lim_{n\to\infty} \int_0^T \eta_{n,\ell}(t)\phi(t)\,dt = \phi(t_\ell),
\qquad \ell=1,\dots,m.
\]
Therefore,
\begin{equation}\label{eq:app-values-case1}
\lim_{n\to\infty}\int_0^T \theta_n(t)\phi(t)\,dt
= \lim_{n\to\infty}
\sum_{\ell=1}^m \alpha_\ell \int_0^T \eta_{n,\ell}(t)\phi(t)\,dt
 = 
\sum_{\ell=1}^m \alpha_\ell \phi(t_\ell).
\end{equation}
Define
\[
H_n(s):=\int_s^T \theta_n(r)\,dr.
\]

Set $t_0 := 0$ and
\[
\beta_\ell := \sum_{r=\ell}^m \alpha_r,
\qquad \ell = 1,\dots,m.
\]

Then, for all $s \in (0,T)\setminus\{t_1,\dots,t_m\}$, we have 
$H_n(s) \longrightarrow H(s)$ as $n\to\infty$,
where the limit function $H$ is given by
\[
H(s)
:=
\sum_{\ell=1}^m \beta_\ell \mathbf{1}_{(t_{\ell-1},t_\ell]}(s).
\]
Moreover,
\[
|H_n(s)| \le M := \sum_{\ell=1}^m |\alpha_\ell|,
\qquad s \in [0,T],\ n \ge 1.
\]
For $k \in \mathbb{Z}$, define
\[
u_{n,k}
:=
\int_0^T p_k(s)\bigl(e^{H_n(s)} - 1\bigr)\,ds,
\qquad
u_k
:=
\int_0^T p_k(s)\bigl(e^{H(s)} - 1\bigr)\,ds.
\]

Since $H_n(s)\to H(s)$ pointwise and $|H_n(s)|\le M$, we have
$
|e^{H_n(s)} - 1| \le C_M$
for some constant $C_M>0$. Hence, by dominated convergence,
$
u_{n,k} \longrightarrow u_k$ as $n \to \infty$,
for any fixed $k$.
Using the explicit form of $H$, we obtain
\[
u_k
=
\sum_{\ell=1}^m (e^{\beta_\ell}-1)
\int_{t_{\ell-1}}^{t_\ell} p_k(s)\,ds.
\]
Define
\[
p_{k,\ell}
:=
\int_{t_{\ell-1}}^{t_\ell} p_k(s)\,ds,
\]
so that
\[
u_k = \sum_{\ell=1}^m (e^{\beta_\ell}-1)p_{k,\ell}.
\]
Hence,
\[
1+u_k
=
1-\sum_{\ell=1}^m p_{k,\ell}
+
\sum_{\ell=1}^m p_{k,\ell}e^{\beta_\ell}.
\]
Recalling the definition of the finite-dimensional Laplace functional,
we conclude that
\begin{equation}\label{eq:uk-limit}
\sum_{k \in \mathbb{Z}} \log(1+u_k)
=
\Lambda_{t_1,\dots,t_m}(\alpha).
\end{equation}
We now pass to the limit inside the infinite sum.
From the bound $|H_n(s)|\le M$, we obtain
\[
|u_{n,k}|
\le
C_M\,\mathbb{P}(A_k \in (0,T]),
\qquad
|u_k|
\le
C_M\,\mathbb{P}(A_k \in (0,T]).
\]
Moreover, since $H_n(s)\ge -M$ and $H(s)\ge -M$,
\[
1+u_{n,k} \ge e^{-M},
\qquad
1+u_k \ge e^{-M}.
\]
Thus the logarithm is Lipschitz on $[e^{-M},\infty)$, and there exists
$L_M > 0$ such that
\[
|\log(1+u_{n,k}) - \log(1+u_k)|
\le
L_M\,|u_{n,k}-u_k|.
\]
Furthermore,
\[
|u_{n,k}-u_k|
\le
2C_M\,\mathbb{P}(A_k \in (0,T]).
\]
Since
\[
\sum_{k \in \mathbb{Z}} \mathbb{P}(A_k \in (0,T]) < \infty,
\]
we can apply dominated convergence for series to conclude that
\[
\sum_{k \in \mathbb{Z}} \log(1+u_{n,k})
\longrightarrow
\sum_{k \in \mathbb{Z}} \log(1+u_k).
\]
Together with \eqref{eq:uk-limit}, this yields
\begin{equation}\label{eq:app-lambda-limit-case1}
\lim_{n\to\infty}\Lambda_T(\theta_n,0) = \Lambda_{t_1,\dots,t_m}(\alpha).
\end{equation}
Finally, by the definition of $\mathcal V_T(\phi)$,
\begin{equation}
\mathcal V_T(\phi)
\ge
\int_0^T \theta_n(t)\phi(t)\,dt - \Lambda_T(\theta_n,0).
\end{equation}
Passing to the limit and using \eqref{eq:app-values-case1} and
\eqref{eq:app-lambda-limit-case1}, we obtain
\begin{equation} \label{eq:ineqV}
\mathcal V_T(\phi)
\ge
\sum_{\ell=1}^m \alpha_\ell \phi(t_\ell)
-
\Lambda_{t_1,\dots,t_m}(\alpha).
\end{equation}

\medskip

\noindent {\it Case 2: $t_m = T$.}
Set $\gamma := \alpha_m$. If $m=1$, set $\theta_n \equiv 0$. If $m\ge 2$, define
\begin{equation}
    \theta_n(t) := \sum_{\ell = 1}^{m-1} \alpha_\ell\eta_{n,\ell}(t),
\end{equation}
where $\eta_{n,\ell}$ is defined as in Case 1. Then, 
\begin{equation}\label{eq:app-values-case2}
  \lim_{n\to\infty}  \int_0^T \theta_n(t)\phi(t) \, dt + \gamma \phi(T) = \sum_{\ell = 1}^{m-1}\alpha_\ell \phi(t_\ell) + \alpha_m \phi(T) = \sum_{\ell=1}^m \alpha_\ell \phi(t_\ell).
\end{equation}
Again define 
\begin{equation}
    H_n(s):=\int_s^T \theta_n(r) \, dr. 
\end{equation}
Set $t_0 := 0$ and 
\begin{equation}
    \beta_\ell := \sum_{r=\ell}^m \alpha_r, \qquad \ell=1,\dots,m.
\end{equation}
Then, for all $s\in (0,T)\setminus\{t_1,\dots,t_{m-1}\}$, 
\begin{equation}
    \lim_{n\to\infty} \gamma + H_n(s) = H(s),
\end{equation}
with the limit function $H$ given by
\begin{equation}
    H(s):= \sum_{\ell=1}^m \beta_\ell \mathbf 1_{(t_{\ell -1}, t_\ell]}(s).
\end{equation}
Indeed, on the last interval $(t_{m-1}, T]$, the limiting contribution of the integral term is zero, while the terminal term gives 
 $
    \gamma = \alpha_m = \beta_m.$
Moreover, $\gamma + H_n(s)$ is uniformly bounded in $n$ and $s$. 
Therefore, the same dominated-convergence argument of Case 1, with $H_n$ replaced by $\gamma + H_n$ gives 
\begin{equation}
   \lim_{n\to\infty} \Lambda_T(\theta_n, \gamma) = \Lambda_{t_1,\dots,t_m}(\alpha).
\end{equation}
Together with \eqref{eq:app-values-case2}, and using the definition of $\mathcal V_T$, we obtain the counterpart of \eqref{eq:ineqV} for $t_m=T.$

\medskip
For both $t_m<T$ and $t_m=T$, we have now proved \eqref{eq:app-values-case2}.
Taking the supremum over $\alpha \in \R^m$ yields
\begin{equation}
    \mathcal{V}_T(\phi) \ge I_{t_1,\dots,t_m}\bigl(\phi(t_1),\dots,\phi(t_m)\bigr).
\end{equation}
Finally, taking the supremum over all $m\ge 1$ and all $0<t_1<\cdots,t_m\le T$, we obtain
$$
\mathcal V_T (\phi) \ge I_T(\phi), 
$$
which is the desired converse inequality. 
\section{Proof of Proposition \ref*{prop:ldp_cost_overflowpath}}\label{app:overflow}

In this section we prove Proposition \hyperlink{prop-overflowpath-target}{\ref*{prop:ldp_cost_overflowpath}}. 

\begin{proof}

Fix \(t\in(0,T]\) and let
$
x:=\phi_t^*(t).
$
By construction, $
\phi_t^*(t) = \E^{\theta_t}[A(t)] = \Lambda_t'(\theta_t).$
We prove
\[
I_T(\phi_t^*)=I_A[0,t](x).
\]

First, since the projective representation of \(I_T\) contains the one-dimensional projection at time \(t\), we have
\[
I_T(\phi_t^*) \ge I_A[0,t]\bigl(\phi_t^*(t)\bigr) = I_A[0,t](x).
\]
It remains to prove the reverse inequality. Recall that
\[
I_T(\phi) = \sup_{m\ge1} \sup_{0<t_1<\cdots<t_m\le T} I_{t_1,\ldots,t_m} \bigl( \phi(t_1),\ldots,\phi(t_m) \bigr),
\]
where
\[
I_{t_1,\ldots,t_m}(y) = \sup_{\alpha\in\mathbb R^m} \left\{ \sum_{\ell=1}^m \alpha_\ell y_\ell - \Lambda_{t_1,\ldots,t_m}(\alpha) \right\}.
\]
Fix \(0<t_1<\cdots<t_m\le T\) and
\(\alpha=(\alpha_1,\ldots,\alpha_m)\in\mathbb R^m\). Define
\[
F = \sum_{\ell=1}^m \alpha_\ell A(t_\ell).
\]
Then, since \(\phi_t^*(u)=\E^{\theta_t}[A(u)]\),
\[
\E^{\theta_t}[F] = \sum_{\ell=1}^m \alpha_\ell \phi_t^*(t_\ell).
\]
We use Gibbs' inequality in the form
\[
\E^{\theta_t}[F]-\log \E[e^F] \le \E^{\theta_t} \left[ \log \frac{d\mathbb P^{\theta_t}}{d\mathbb P} \right].
\]
Therefore,
\[ \sum_{\ell=1}^m \alpha_\ell \phi_t^*(t_\ell) - \Lambda_{t_1,\ldots,t_m}(\alpha) \le \E^{\theta_t}\left[ \log\frac{d\mathbb P^{\theta_t}}{d\mathbb P} \right].
\]
Taking the supremum over \(\alpha\in\mathbb R^m\) yields
\[ I_{t_1,\ldots,t_m} \bigl(\phi_t^*(t_1),\ldots,\phi_t^*(t_m)\bigr)\le\E^{\theta_t}\left[\log\frac{d\mathbb P^{\theta_t}}{d\mathbb P}\right].
\]

Now compute the right-hand side. Since
\[
\frac{d\mathbb P^{\theta_t}}{d\mathbb P} =\exp\left(\theta_tA(t)-\Lambda_t(\theta_t)\right),
\]
we have
\[
\E^{\theta_t}\left[\log\frac{d\mathbb P^{\theta_t}}{d\mathbb P}
\right] = \E^{\theta_t} \left[ \theta_tA(t)-\Lambda_t(\theta_t)\right]=\theta_t\E^{\theta_t}[A(t)]-\Lambda_t(\theta_t).
\]
Using
$
\E^{\theta_t}[A(t)]=\phi_t^*(t)= x,
$
we obtain
\[
\E^{\theta_t}\left[\log\frac{d\mathbb P^{\theta_t}}{d\mathbb P}\right]=\theta_tx-\Lambda_t(\theta_t).
\]
Since \(\theta_t\) is chosen such that
$
\Lambda_t'(\theta_t)=x,$
it is the maximizer in the Legendre transform defining \(I_A[0,t](x)\). Therefore,
$\theta_tx-\Lambda_t(\theta_t)= I_A[0,t](x).$
Thus, for every finite collection \(0<t_1<\cdots<t_m\le T\),
\[
I_{t_1,\ldots,t_m}\bigl(\phi_t^*(t_1),\ldots,\phi_t^*(t_m)\bigr)
\le I_A[0,t](x).
\]
Taking the supremum over all finite collections of times gives
\[
I_T(\phi_t^*)\le I_A[0,t](x).
\]

Upon combining the lower and upper bounds, we conclude that
\[
I_T(\phi_t^*)=I_A[0,t](x)=I_A[0,t]\bigl(\phi_t^*(t)\bigr).
\]
This completes the proof.
    
\end{proof}
\bibliographystyle{alpha}
\bibliography{preamble/BigBib}
\end{document}

%% file: preamble/packages.tex
\usepackage{graphicx}

\usepackage{mathscinet}

\usepackage[english]{babel}
\usepackage{csquotes}

\usepackage{microtype} 
\usepackage[title,titletoc,page]{appendix}
\usepackage[a4paper,left=3cm,right=3cm,top=2cm,bottom=4cm,bindingoffset=5mm]{geometry}
\usepackage[colorlinks=true,linkcolor=blue,citecolor=black,linktocpage=true]{hyperref}

\usepackage{amsfonts}
\usepackage{amsmath}
\usepackage{amssymb}
\usepackage{dsfont}
\usepackage{mathtools}
\mathtoolsset{showonlyrefs}

\usepackage{thmtools}
\usepackage[shortlabels]{enumitem}

\usepackage{version}
\usepackage[obeyFinal]{todonotes}
\setuptodonotes{inline}
\usepackage{cancel}
\usepackage[normalem]{ulem}

\numberwithin{equation}{section}

%% file: preamble/shortcuts.tex
\usepackage{bbm}
\usepackage{bm}
\usepackage[normalem]{ulem}
\usepackage{cancel}

\newcommand{\E}{\mathbb{E}}
\newcommand{\Pbb}{\mathbb{P}}
\newcommand{\R}{\mathbb{R}}
\newcommand{\Z}{\mathbb{Z}}
\newcommand{\1}{\mathbf{1}}

\renewcommand{\P}{\mathbb{P}}

%% file: preamble/theoremEnvironments.tex
\newtheorem{theorem}{Theorem}[section]
\newtheorem{lemma}[theorem]{Lemma}

\newtheorem{proposition}[theorem]{Proposition}

\theoremstyle{definition}
\newtheorem{definition}[theorem]{Definition}
\newtheorem{remark}[theorem]{Remark}
\newtheorem{assumption}[theorem]{Assumption}